\documentclass[11pt]{article}

\usepackage[latin1]{inputenc} 
\usepackage{amsthm,amsmath,amssymb} 
\usepackage{graphicx}
\usepackage{color}
\usepackage{float}

\usepackage{enumitem} 

\usepackage[hidelinks]{hyperref}
\usepackage{geometry}
\usepackage{hyperref, enumerate}
\usepackage{rotating, tikz}
\usetikzlibrary{decorations.pathmorphing,decorations.pathreplacing,decorations.shapes}

\setlist[enumerate]{topsep=0pt,itemsep=-1ex,partopsep=1ex,parsep=1ex}

\theoremstyle{plain}
\newtheorem{theorem}{Theorem}
\newtheorem{lemma}[theorem]{Lemma}
\newtheorem{claim}[theorem]{Claim}
\newtheorem{proposition}[theorem]{Proposition}

\newtheorem{conjecture}[theorem]{Conjecture}

\newcommand{\floor}[1]{\lfloor #1 \rfloor}

\title{On the length of directed paths in digraphs}
\author{Yangyang Cheng$^{a,}$\unskip\thanks{\emph{E-mail address:} Yangyang.Cheng@maths.ox.ac.uk} \unskip\thanks{Supported by a PhD studentship of ERC Advanced Grant 883810.},
Peter Keevash$^{a,}\unskip\thanks{Supported by ERC Advanced Grant 883810.}$
\\ \\[.5em]
{\small $^a$Mathematical Institute, University of Oxford, UK}\\}
\date{}

\begin{document}
\maketitle

\begin{abstract}
Thomass\'{e} conjectured the following strengthening of the well-known Caccetta-Haggkvist Conjecture:
any digraph with minimum out-degree $\delta$ and girth $g$ contains a directed path of length $\delta(g-1)$.
Bai and Manoussakis \cite{Bai} gave counterexamples to Thomass\'{e}'s conjecture for every even $g\geq 4$. 
In this note, we first generalize their counterexamples to show
that Thomass\'{e}'s conjecture is false for every $g\geq 4$. 
We also obtain the positive result that any digraph with minimum 
out-degree $\delta$ and girth $g$ contains a directed path of $2(1-\frac{2}{g})$. 
For small $g$ we obtain better bounds, e.g.~for $g=3$ we show that
oriented graph with minimum out-degree $\delta$ contains a directed path of length $1.5\delta$.
Furthermore, we show that each $d$-regular digraph with  girth $g$ 
contains a directed path of length $\Omega(dg/\log d)$. 
Our results give the first non-trivial bounds for these problems.
\end{abstract}


\section{Introduction}

The Caccetta-Haggkvist Conjecture \cite{caccetta1978minimal}
states that any digraph on $n$ vertices with minimum out-degree $\delta$ 
contains a directed cycle of length at most $\lceil n/\delta \rceil$;
it remains largely open (see the survey \cite{Sul}).
A stronger conjecture proposed by Thomass\'{e} (see \cite{Bang,Sul}) states that 
any digraph with minimum out-degree $\delta$ and girth $g$ contains a directed path of length $\delta(g-1)$.
Bai and Manoussakis \cite{Bai} gave counterexamples to Thomass\'{e}'s conjecture for every even $g\geq 4$. 
The conjecture remains open for $g=3$, which in itself was highlighted as an unsolved problem
in the textbook  \cite{GT}. 

\begin{conjecture}
Any oriented graph with minimum out-degree $\delta$ contains a directed path of length $2\delta$.
\end{conjecture}

In this note, we first generalize the counterexamples to show
that Thomass\'{e}'s conjecture is false for every $g\geq 4$. 

\begin{proposition}\label{counter}
For every $g\geq 2$ and $\delta \geq 1$ there exists a digraph $D$ with girth $g$ and $\delta^+(D)\geq \delta$
such that any directed path has length at most $\frac{g \delta +g-2}{2}$ if $g$ is even or $\frac{(g+1)\delta+g-3}{2}$ if $g$ is odd.
\end{proposition}

In the positive direction, when $g$ is large we can find a directed path of length close to $2\delta$.

\begin{theorem}\label{thm1}
Every digraph $D$ with girth $g$ and $\delta^+(D)\geq \delta$ contains a directed path of length $2\delta(1-\frac{1}{g})$.
\end{theorem}

For the cases $g=3$ or $g=4$, we have the following better bounds.

\begin{theorem}\label{thm2}
Every oriented graph $D$ with $\delta^+(D)\geq \delta$ contains a directed path of length $1.5\delta$. 
Every digraph $D$ with $\delta^+(D)\geq \delta$ and girth $g\geq 4$ contains a directed path of length $1.6535\delta$.
\end{theorem}

Finally, we consider the additional assumption of approximate regularity,
under which a standard application of the Lov\'asz Local Lemma gives much better bounds,
We call a digraph $(C,d)$-\emph{regular} if $d^+(v)\geq d$ and $d^-(v)\leq Cd$ for each vertex $v$. 

\begin{theorem}\label{ldp}
For every $C>0$ there exists $c>0$ such that if $D$ is a $(C,d)$-regular digraph with girth $g$
then $D$ contains a directed path of length at least $cdg/\log d$.
\end{theorem}

\subsection{Notation}

We adopt standard notation as in \cite{Bang}.
A \emph{digraph} $D$ is defined by a vertex set $V(D)$ and arc set $A(D)$, which is a set of ordered pairs in $V(D)$.
An \emph{oriented graph} is a digraph where we do not allow $2$-cycles $\{(x,y),(y,x)\}$,
i.e.~it is obtained from a simple graph by assigning directions to the edges.
For each vertex $v\in D$ and any vertex set $S\subseteq V(D)$, 
let $N^+(v,S)$ be the set of out-neighbours of $v$ in $S$ and let $d^+(v,S)=|N^+(v,S)|$. 
If $S=V(D)$, then we simply denote $d^+(v,S)$ by $d^+(v)$. 
If $H$ is an induced subgraph of $D$, then we define $d^+(v,H)=d^+(v,V(H))$ for short. 
We let $\delta^+(D) = \min_v d^+(v)$ be the minimum out-degree of $D$.
Indegree notation is similar, replacing $+$ by $-$.

For every vertex set $X\subseteq V(D)$, 
let $N^+(X)$ be the set of vertices that are not in $X$ but are out-neighbours of some vertex in $X$. 
For every two vertex sets $A, B$ of $V(D)$, let $E(A,B)$ be the set of arcs in $A(D)$ with tail in $A$ and head in $B$.
A digraph $D$ is \emph{strongly-connected} if for every ordered pair of vertices $u,v\in V(D)$
there exists a directed path from $u$ to $v$. 

The \emph{girth} $g(D)$ of $D$ is the minimum length of a directed cycle in $D$ (if $D$ is acyclic we define $g(D)=\infty$). 
We write $\ell(D)$ for the maximum length of a directed path in $D$. 

\section{Construction}

We start by constructing counterexamples to  Thomass\'{e}'s conjecture 
for every $g \ge 4$, as stated in Proposition \ref{counter}.
Suppose that $D$ is a digraph with $d^+(v)=\delta$ for each vertex $v\in V(D)$. 
For each $k\geq 1$, we define the \emph{k-lift} operation on some fixed vertex $v$ as follows: 
we delete all arcs with tail $v$, add $k-1$ disjoint sets of $\delta$ new vertices $U_{v,1},...,U_{v,k-1}$ to $D$,
write $U_{v,0}:=\{v\}$, $U_{v,k}:=N^+(v)$ and add arcs so that
$U_{v,i-1}$ is completely directed to $U_{v,i}$ for $1 \le i \le k$.
(For example, a $1$-lift does not change the digraph.)
We note that any lift preserves the property that all out-degrees are $\delta$.

Write $\overrightarrow{K}_{\delta+1}$ for the complete directed graph on $\delta+1$ vertices.
Our construction is $D_{a,b} := \overrightarrow{K}_{\delta+1}^{\uparrow}(a,b,\dots,b)$ for  some integer $1\leq a \leq b$, 
meaning that starting from $\overrightarrow{K}_{\delta+1}$, we $a$-lift some vertex $v_1$ and $b$-lift all the other vertices.

\begin{claim}\label{lift}
The girth of $D_{a,b}$ is $a+b$ and the longest path has length $\delta b+a-1$.
\end{claim}
\begin{proof}
Let $C$ be any directed cycle in $D_{a,b}$.
By construction, we can decompose $A(C)$ into directed paths of the form $v_iu_1\cdots u_tv_j$ 
such that $u_j\in U_{v_i,j}$ for $1\leq j\leq t$ where $t=a-1$ if $i=1$ and $t=b-1$ if $i\geq 2$;
we call such paths `full segments' and their subpaths `segments'.
A directed cycle contains at least two full segments, so its length is at least $a+b$ since $a\leq b$.
It is also easy to see $D_{a,b}$ does contain a directed cycle of length $a+b$. 

Now suppose $P$ is a directed path in $D_{a,b}$ of maximum length.
Similarly, we can decompose $E(P)$ into segments, where all but the first and the last are full,
and if there are $\delta+1$ segments then at most one of the first and the last is full.
Each segment has length at most $b$, except that the one starting from $v_1$ has length at most $a$,
so $P$ has length at most $\delta b+a-1$.
\end{proof}

Proposition \ref{counter} follows from Claim \ref{lift} 
by taking $a=b=\frac{g}{2}$ for $g$ even 
or $a=\frac{g-1}{2}$ and $b=\frac{g+1}{2}$ for $g$ odd.

\section{The key lemma}

Here we show that Theorems \ref{thm1} and \ref{thm2} follow directly from
known results on the Caccetta-H\"{a}ggkvist conjecture and the following key lemma.

\begin{lemma}\label{lem1}
If $D$ is an oriented graph with $\delta^+(D) \geq \delta$ 
then $D$ either contains a directed path of length  $2\delta$ 
or an induced subgraph $S$ such that $|S|\leq \delta$ and $\delta^+(S) \geq 2\delta-\ell(D)$.
\end{lemma}

We use the following bounds on Caccetta-H\"{a}ggkvist in general by  Chv\'{a}tal and Szemer\'{e}di \cite{Chvatal} 
and in the case of directed triangles by  Hladk\'{y}, Kr\'{a}l, and Norin \cite{Norin}.

\begin{theorem}\label{thm0}
Every digraph $D$ with order $n$ and $\delta^+(D) \geq \delta$ contains a directed cycle of length at most $\lceil \frac{2n}{\delta+1}\rceil$.
\end{theorem}

\begin{theorem}\label{thm00}
Every oriented graph with order $n$ and minimum out-degree $0.3465n$ contains a directed triangle.
\end{theorem}

Now we deduce Theorems \ref{thm1} and \ref{thm2}, assuming the key lemma.

\begin{proof}[Proof of Theorem \ref{thm1}]
Suppose that $D$ is an oriented graph with $\delta^+(D)\geq \delta$ and girth $g$. 
By Lemma \ref{lem1}, $D$ contains a directed path of length $2\delta$ or an induced subgraph $S$ with $|S|\leq \delta$ and $\delta^+(S) \geq 2\delta-\ell(D)$. We assume the latter case holds.
According to Theorem \ref{thm0}, $S$ contains a directed cycle of length at most $\frac{2\delta}{2\delta-\ell(D)+1}$. 
Therefore, $g\leq \frac{2\delta}{2\delta-\ell(D)+1}$, so $\ell(D) \geq 2\delta(1-\frac{1}{g})+1\geq 2\delta(1-\frac{1}{g})$.
\end{proof}

\begin{proof}[Proof of Theorem \ref{thm2}]
First, suppose that $D$ is an oriented graph with $\delta^+(D)\geq \delta$. 
By Lemma \ref{lem1}, either $D$ contains a directed path of length $2\delta$ 
or $D$ contains an induced subgraph $S$ such that $|S|\leq \delta$ and $\delta^+(S) \geq 2\delta-\ell(D)$. 
Since $D$ is oriented, for some vertex $b\in S$, we have $d^+(b,S)\leq \frac{|S|-1}{2}$, 
which means that $\delta^+(S)\leq \frac{|S|-1}{2} \leq \frac{\delta-1}{2}$ and so $\ell(D) \geq 2\delta-\delta^+(S)\geq \frac{3}{2}\delta$.
Similarly, if $D$ has girth at least $4$ then substituting the bound $\delta^+(S)< 0.3465 \delta$ from Theorem \ref{thm00}
we obtain $\ell(D)>1.6535\delta$.
\end{proof}

In fact, by Lemma \ref{lem1}, any improved bound towards the Caccetta-H\"{a}ggkvist conjecture 
can be used to get a better bound for $\ell(D)$ when $\delta^+(D)\geq \delta$ and girth $g$. For example, the main result in \cite{shen2002caccetta} will give the bound $\ell(D) \ge (2-\frac{1}{g-73})\delta$. 
The Caccetta-H\"{a}ggkvist conjecture itself would imply $\ell(D) \ge (2-\frac{1}{g})\delta$.

\section{Proof of the key lemma}

Suppose that $D$ is an oriented graph with $\delta^+(D)\geq \delta$ 
and no directed path of length $2\delta$.
We can assume that $D$ is strongly-connected,
as there is a strong component of $D$ with minimum out-degree at least $\delta$.
By deleting arcs, we can also assume that all out-degrees are exactly $\delta$.
Note that $|V(D)|\geq 2\delta+1$, since $D$ is oriented and $\delta^+(D)\geq \delta$. 

\begin{claim} \label{clm1}
$D$ does not  contain two disjoint directed cycles of length at least $\delta+1$. 
\end{claim}
\begin{proof}
Suppose on the contrary that $C_1$ and $C_2$ are two such cycles. By strong connectivity, 
there exists a path $P$ from $u_1\in C_1$ to $u_2\in C_2$ with $V(P)$ internally disjoint from $V(C_1) \cup V(C_2)$. 
Writing $u_1 u_1^{\prime}$ for the out-arc of $u_1$ in $C_1$ and $u_2^{\prime} u_2$ for the in-arc of $u_2$ in $C_2$, 
the path $\{C_1-u_1 u_1^{\prime} \}+P+\{C_2-u_2^{\prime} u_2 \}$ has length at least $2\delta+1$, a contradiction.
\end{proof}

Now let $P=v_0v_1 \cdots v_{\ell(D)}$ be a directed path of maximum length, where $\ell(D)<2\delta$.
By maximality of $P$, the out-neighbours $N^+(v_{\ell(D)})$ of $v_{\ell(D)}$ must lie on $P$. 
Let $v_a \in N^+(v_{\ell(D)})$ such that the index $a$ is minimum among all the out-neighbours of $v_{\ell(D)}$. 
Thus $C=v_av_{a+1}\cdots v_{\ell(D)}v_a$ is a directed cycle; we call $|C|$ the \emph{cycle bound} of $P$.
For future reference, we record the consequence
\begin{equation} \label{cyc}
\ell(D) \ge g(D) \text{ for any digraph } D.
\end{equation}

Choose $P$ such that the cycle bound of $P$ is also maximum subject to that $P$ is a directed path of length $\ell(D)$. 
Clearly $a\neq 0$, otherwise using $|V(D)|\geq 2\delta+1$ and strong connectivity, 
we can easily add one more vertex to $C$ and get a longer path, contradiction.

\begin{claim} \label{clma}
Every vertex in $N^+(v_{a-1})$ must be on $P$.
\end{claim}
\begin{proof}
Suppose on the contrary that there exists an out-neighbour $w_1$ of $v_{a-1}$ such that $w_1\in V(D)\setminus V(P)$. 
Let $D_1$ be the induced graph of $D$ on $V(D)\setminus V(P)$. We extend the vertex $w_1$ to a maximal directed path $P_1=w_1w_2\cdots w_m$ in $D_1$. 
Since $P_1$ is maximal in $D_1$, all the out-neighbours of $w_m$ must be on $V(P)\cup V(P_1)$, see Figure \ref{fig1}(a).

\begin{figure}[h]
\centering
\begin{tikzpicture}[>=stealth, scale=0.8]
\draw (4,15.2) node[anchor=center] {$\ldots$};
\draw[decorate, segment length=6.5mm, line width=1pt, ->] 
    (2,15)-- (3,15);
\draw[decorate, segment length=6.5mm, line width=1pt, ->] 
    (2.9,15)-- (8,15);
\draw[decorate, segment length=6.5mm, line width=1pt] 
    (7.9,15)-- (9.4,15);
\draw[segment length=6.5mm, line width=1pt, dashed, ->] 
    (6.5,13) .. controls (5.5,13.3) and (4.5,14) ..(3.5,14.95);
\draw[decorate, segment length=6.5mm, line width=1pt] 
    (5,15)-- (6.5,13);
\draw[decorate, segment length=6.5mm, line width=1pt, ->] 
    (5.36,14.52)-- (5.46,14.4);
\draw (6,14.2) node[anchor=center] {
    \begin{turn}{-20}
        $\ddots$
    \end{turn}};
\draw[segment length=6.5mm, line width=1pt, dashed, ->] 
    (6.5,13) -- (6.8,14.95);
\draw (6.7,15.2) node[anchor=center] {$\ldots$};  \fill (6.5,13) circle (0.07cm); 
\fill (5,15) circle (0.07cm);
\draw[segment length=6.5mm, line width=1pt] (5.4,15) .. controls (6.4,16.3) and (8.4,16.3) .. (9.4,15);
\draw[decorate, segment length=6.5mm, line width=1pt, ->] 
    (7.3,15.98)-- (7.2,15.98);
\fill (5.4,15) circle (0.07cm);
\fill (9.4,15) circle (0.07cm);
\fill (2,15) circle (0.07cm);
\draw (2,14.7) node[anchor=center] {$v_0$};
\draw (4.7,14.7) node[anchor=center] {$v_{a-1}$};
\draw (5.6,14.7) node[anchor=center] {$v_{a}$};
\draw (9.4,14.7) node[anchor=center] {$v_{l(D)}$};
\draw (2,16.1) node[anchor=center] {$(a)$};
\end{tikzpicture}
\hspace{3em}
\begin{tikzpicture}[>=stealth, scale=0.8]
\draw (4,15.2) node[anchor=center] {$\ldots$};
\draw[decorate, segment length=6.5mm, line width=1pt, ->] 
    (2,15)-- (3,15);
\draw[decorate, segment length=6.5mm, line width=1pt, ->] 
    (2.9,15)-- (7.5,15);
\draw[decorate, segment length=6.5mm, line width=1pt] 
    (7.4,15)-- (9.4,15);
\fill (5,15) circle (0.07cm);
\draw[segment length=6.5mm, line width=1pt] (5.4,15) .. controls (6.4,16.3) and (8.4,16.3) .. (9.4,15);
\draw[decorate, segment length=6.5mm, line width=1pt, ->] 
    (7.3,15.98)-- (7.2,15.98);
\draw[segment length=6.5mm, line width=1pt] (5,15) .. controls (5.6,13) and (8,13) .. (8.6,15);
\draw[decorate, segment length=6.5mm, line width=1pt, ->] 
    (6.9,13.51)-- (7,13.525);
\draw[segment length=6.5mm, line width=1pt, dashed,->] (8.2,15) .. controls (7,13.5) and (4.7,13.5) .. (3.5,14.95);
\fill (5.4,15) circle (0.07cm);
\fill (9.4,15) circle (0.07cm);
\fill (2,15) circle (0.07cm);
\draw (2,14.7) node[anchor=center] {$v_0$};
\draw (4.6,14.7) node[anchor=center] {$v_{a-1}$};
\draw (5.5,14.7) node[anchor=center] {$v_{a}$};
\draw (9.4,14.7) node[anchor=center] {$v_{l(D)}$};
\draw (8,15.3) node[anchor=center] {$b$};
\draw (8.7,15.3) node[anchor=center] {$b^+$};
\draw (2,16.1) node[anchor=center] {$(b)$};
\fill (8.2,15) circle (0.07cm);
\fill (8.6,15) circle (0.07cm);
\draw (6.7,15.2) node[anchor=center] {$\ldots$}; 
\end{tikzpicture}
\caption{Illustrations for the proofs of Claims \ref{clma} and \ref{claim2}.}
\label{fig1}
\end{figure}
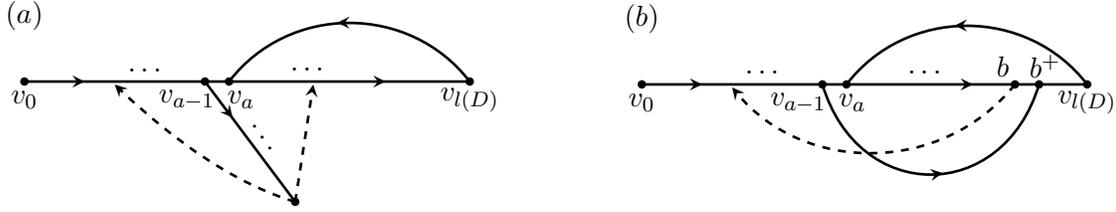

We cannot have $w\in N^+(u_m)$ such that $w\in V(C)$. Indeed, writing $w^-$ for the in-neighbour of $w$ in $C$,
the directed path $P^{\prime}=v_0\ldots v_{a-1}P_1w+(C-w^-w)$ would be longer than $P$, a contradiction.
Thus we conclude that $N^+(w_m)\subseteq V(P_1)\cup \{v_0,...,v_{a-1}\}$. 
Choose a vertex $z\in N^+(w_m)$ that has the largest distance to $w_m$ on the path $P_2=v_0\ldots v_{a-1}w_1 \ldots w_m$. 
Then $P_2\cup w_mz$ contains a cycle $C_1$ of length at least $\delta +2$. 
Now $C_1$ and $C$ are two disjoint directed cycles of length at least $\delta+2$, 
which contradicts Claim \ref{clm1}.
\end{proof}

Let $A=N^+(v_{a-1})\cap \{v_0,\ldots, v_{a-1}\}$ and $B=N^+(v_{a-1})\cap V(C)$. 
Also, let $B^-=\{u: u\in V(C), uv\in A(C)$ for some $v \in B\}$.

\begin{claim} \label{claim2}
$N^+(B^-)\subseteq V(C)$.
\end{claim}

\begin{proof}
Suppose not, then there exists a vertex $w\in V(D)\setminus V(C)$ such that $bw\in A(D)$ for some $b\in B^-$. 
By definition of $B$, there exists some vertex $b^+\in B$ such that $v_{a-1}b^+\in A(D)$ and $bb^+ \in A(C)$. 
We cannot have $w\in V(D)\setminus V(P)$, as then the path $v_0v_1\ldots v_{a-1}b^++(C-bb^+)+bw$ has  length $\ell(D)+1$, a contradiction.

It remains to show that we cannot have $w\in V(P)\setminus V(C)$. Suppose that we do, with $w=v_i$ for some $0\leq i\leq a-1$. 
Then the cycle $v_iv_{i+1}\ldots v_{a-1}b^++ (C-bb^+)+bv_i$ is longer than $C$. 
However, $P_1=v_0\ldots v_{a-1}b^++(C-bb^+)$ has length $\ell(D)$ and cycle bound larger than $P$, which contradicts our choice of $P$,
see Figure \ref{fig1}(b).
\end{proof}

Now let $S$ be the induced digraph of $D$ on $B^-$. 
Fix $x\in B^-$ with $N^+_{S}(x)=\delta^+(S)$. 
Then $N^+(x)\subseteq V(C)$ by Claim \ref{claim2}.
As $|N^+(x)|=\delta$ we deduce  $|C|\geq |S|-\delta^+(S)+\delta$.

Note that $|P|\geq |A|+1+|C|\geq |A|+1+|B|-\delta^+(S)+\delta$,
as $|S|=|B^-|=|B|$ and $A\subseteq \{v_0,\ldots,v_{a-1}\}$.
But $|A|+|B|=|N^+(v_{a-1})|=\delta$, so
$\ell(D)=|P|\geq 2\delta+1-\delta^+(S)$ and $\delta^+(S)\geq 2\delta+1-\ell(D)$.

This completes the proof of Lemma \ref{lem1}.

\section{Long directed paths in almost-regular digraphs}

In this section, we prove Theorem \ref{ldp}. We start by stating some standard
probabilistic tools (see \cite{PT}). We use the following version of Chernoff's inequality.
\begin{lemma}
Let $X_1,\ldots,X_n$ be independent Bernoulli random variables 
with $\mathbb{P}[X_i=1]=p_i$ and $\mathbb{P}[X_i=0]=1-p_i$ for all $i\in [n]$. 
Let $X=\sum_{i=1}^nX_i$ and $E[X]=\mu$. 
Then for every $0<a<1$, we have $$\mathbb{P}[|X-\mu|\geq a\mu]\leq 2e^{-a^2\mu/3}.$$
\end{lemma}

We will also use the following version of Lov\'{a}sz Local Lemma.
\begin{lemma} \label{LLL}
Let $A_1,\dots,A_n$ be a collection of events in some probability space.
Suppose that each  $\mathbb{P}[A_i]\leq p$ and each $A_i$ is
mutually independent of a set of all the other events $A_j$ but at most $d$,
where $ep(d+1)<1$. Then $\mathbb{P}[\cap_{i=1}^n \overline{A_i}]>0$.
\end{lemma}

Next we deduce the following useful partitioning lemma.

\begin{lemma}\label{par}
For every $C>0$ there exists $c'>0$ such that 
for any positive integer $d$ with $t:=\floor{c'd/\log d} \ge 1$,
for any $(C,d)$-regular digraph $D$
there exists a partition of $V(D)$ into $V_1\cup \cdots \cup V_t$ 
such that $||V_i|-|V_j||\leq 1$ and $d^+(v,V_j)\geq \frac{\log d}{2c'}$ for each $i,j\in [n]$ and $v\in V_i$.
\end{lemma}
\begin{proof}
We start with an arbitrary partition $U_1\cup \cdots \cup U_{s}$ of $V(D)$ 
where $|U_1|=\cdots=|U_{s-1}|=t$ and $1\leq |U_s|\leq t$, so that $n/t \leq s< n/t+1$. 
We add $t-|U_s|$ isolated `fake' vertices into $U_s$ to make it a set of size $t$.
We consider independent uniformly random permutations 
$\sigma_i = (\sigma_{i,1},\dots,\sigma_{i,t})$ of each $U_i$. 
Now let $V_j=\{\sigma_{1,j}, \ldots, \sigma_{s,j}\}$ for each $1\leq j\leq t$. 
We will show that $V_1\cup \cdots \cup V_t$ (with fake vertices deleted)
gives the required partition with positive probability.

We consider the random variables $X(v,j):=d^+(v,V_j)$ for each $v\in V$ and $j\in [t]$.
Note that each is a sum of independent Bernoulli random variables with $\mathbb{E}[X(v,j)]=d^+(v)/t$.
We let $E_{v,j}$ be the event that $\Big|X(v,j)-\tfrac{d^+(v)}{t}\Big|\geq \tfrac{d^+(v)}{2t}$.
Then $\mathbb{P}[E_{v,j}] \le 2e^{-\frac{d^+(v)}{12 t}}\leq 2e^{-\frac{d}{12 t}}$ by Chernoff's inequality.

Now $E_{v,j}$ is determined by those $\sigma_i$ with $U_i\cap N^+(v)\neq \emptyset$,
so is mutually independent of all but at most $C(dt)^2$ other events $E_{v',j'}$, using $\Delta^-(D)\leq Cd$.
For $c'$ sufficiently small, for example $c' \leq \frac{1}{100\log C}$, 
we get $2e^{-\frac{d}{12t}+1}(C(dt)^2+1)<1$. 
By Lemma \ref{LLL} we conclude that with positive probability no $E_{v,j}$ occurs, 
and so $V_1\cup \cdots \cup V_t$ (with fake vertices deleted) gives the required partition.
\end{proof}

\begin{proof}[Proof of Theorem \ref{ldp}]
Suppose that $D$ is a $(C,d)$-regular digraph with girth $g$.
We will show $\ell(D) \ge \frac{cdg}{\log d}$, where $c=c'/2$ with $c'$ as in Lemma \ref{par}.
As $\ell(D) \ge g(D)$ by \eqref{cyc}, we can assume $c'd/\log d \ge 1$, so $t=\floor{c'd/\log d} \ge 1$.
By  Lemma \ref{par} we can partition $V(D)$ into $V_1\cup \cdots \cup V_t$ 
such that $||V_i|-|V_j||\leq 1$ and each $d^+(v,V_j)\geq \frac{\log d}{2c'}$.
We note that $\frac{\log d}{2c'}>1$ for $c'<0.1$, say. 

Let $P_1$ be a maximal directed path in $D[V_1]$ starting from any vertex $x_1$, ending at some $y_1$.
Then $|P_1|\geq g$ by \eqref{cyc}. By choice of partition, $y_1$ has an out-neighbour $x_2$ inside $D[V_2]$. 
Similarly, we can find a maximal directed path of length at least $g$ inside $D[V_2]$ starting from $x_2$. 
We repeat the process until we find $t$ directed paths $P_1,\ldots,P_t$ of length at least $g$,
that can be connected into a directed path of length at least $tg\geq \frac{c'dg}{2\log d} =  \frac{cdg}{\log d}$. 
This completes the proof.
\end{proof}

\section{Concluding remarks}

We propose the following weaker version of Thomass\'{e}'s conjecture.
\begin{conjecture}
There is some $c>0$ such that $\ell(D) \ge c g(D)\delta^+(D)$ for any digraph $D$.
\end{conjecture}

By Proposition \ref{counter}, the best possible $c$ in this conjecture satisfies $c \le 1/2$.
We do not even know whether it holds for regular digraphs,
or whether $\ell(D)/\delta^+(G) \to \infty$ as $g \to \infty$.

\medskip

\noindent \textbf{Acknowledgments.}
We are grateful to Ant\'{o}nio Gir\~{a}o for helpful discussions.

\bibliographystyle{plain}
\bibliography{Bibte}

\end{document}